\documentclass[a4paper,11pt,leqno]{amsart}

\usepackage{amsmath,amsfonts,amssymb}
\usepackage{pifont,graphicx,epsfig}
\usepackage[bf]{caption2}
\usepackage{epstopdf}

\def\C{\hbox{\font\dubl=msbm10 scaled 1100 {\dubl C}}}
\def\R{\hbox{\font\dubl=msbm10 scaled 1100 {\dubl R}}}

\def\sR{\hbox{\font\dubl=msbm10 scaled 900 {\dubl R}}}
\def\N{\hbox{\font\dubl=msbm10 scaled 1100 {\dubl N}}}
\def\K{\hbox{\font\dubl=msbm10 scaled 1100 {\dubl K}}}
\def\sK{\hbox{\font\dubl=msbm10 scaled 900 {\dubl K}}}
\def\Re{{\rm{Re}}\,}

\def\d{\,{\rm{d}}}

\usepackage{amssymb, amsmath}

\newtheorem{Theorem}{Theorem}
\newtheorem{Corollary}[Theorem]{Corollary}
\newtheorem{Lemma}[Theorem]{Lemma}

\title[Value-distribution of the zeta-function on the critical line]{On the value-distribution of the Riemann zeta-function on the critical line}

\author[Justas Kalpokas, J\"orn Steuding]{Justas Kalpokas, J\"orn Steuding}

\date{May 2009}

\begin{document}

\begin{abstract}
$\mathbb{R}$
 We investigate the intersections of the curve $\sR\ni t\mapsto \zeta({1\over 2}+it)$ with the real axis. We show that if the Riemann hypothesis is true, the mean-value of those real values exists and is equal to $1$. Moreover, we show unconditionally that the zeta-function takes arbitrarily large real values on the critical line. 
\end{abstract}

\maketitle

{\small \noindent {\sc Keywords:} Riemann zeta-function, value-distribution, critical line}
\\ \bigskip

\section{Introduction and statement of the main results}

It is conjectured that the set of values of the Riemann zeta-function $\zeta(s)$ on the critical line $s={1\over 2}+i\R$ is dense in the complex plane. It is even no non-empty open subset of $\C$ known such that $\zeta({1\over 2}+i\R)$ is dense in this set. Some results give evidence in this direction. Bohr \& Courant \cite{bohr} have shown the denseness of $\zeta(\sigma+i\R)$ in the complex plane for any fixed $\sigma\in({1\over 2},1]$ and Selberg (unpublished) has proved that the values taken by an appropriate normalization of the Riemann zeta-function on the critical line are Gaussian normally distributed (the first published proof is due to Joyner \cite{joyner}). However, Garunk\v stis \& Steuding \cite{gs} showed recently that the curve $\R\ni t\mapsto(\zeta({1\over 2}+it),\zeta'({1\over 2}+it))$ is not dense in $\C^2$. In this article we investigate in particular the real points of $\zeta({1\over 2}+it)$ for real $t$. We show that the mean-value of those real values exists and is equal to $1$ provided Riemann's hypothesis is true, and, unconditionally, that the zeta-function takes arbitrarily large real values on the critical line. In his monograph \cite{edwa}, Edwards writes {\it ''... the real part of $\zeta$ has a strong tendency to be positive''} (page 121). We shall explain this phenomenon. For this purpose, we estimate the number of intersections $\zeta({1\over 2}+i\R)$ with any given straight line $e^{i\phi}\R$ and prove asymptotic formulae for the first and second associated discrete moment of those values. This approach yields new information on the value-distribution of the zeta-function on the critical line. 
\medskip

Let $\phi\in[0,\pi)$. The functional equation for $\zeta(s)$ in its asymmetrical form is given by
\begin{equation}\label{feq}
\zeta(s)=\Delta(s)\zeta(1-s),\qquad \mbox{where}\quad \Delta(s):=2^s\pi^{s-1}\Gamma(1-s)\sin({\textstyle{\pi s\over 2}}).
\end{equation}
It is essential for our approach that $\Delta(s)\Delta(1-s)=1$ (which follows directly from the functional equation and will be frequently used in the sequel). Consequently, $\Delta({1\over 2}+it)$ is on the unit circle for any real $t$. Moreover, it follows that
\begin{equation}\label{vieh}
\Phi(t):=\Phi(t;\phi):=\zeta({\textstyle{1\over 2}}+it)-e^{2i\phi}\zeta({\textstyle{1\over 2}}-it)=\zeta({\textstyle{1\over 2}}+it)(1-e^{2i\phi}\Delta({\textstyle{1\over 2}}-it))
\end{equation}
vanishes if and only if either ${1\over 2}+it$ is a zero of the zeta-function or
$$
1=e^{2i\phi}\Delta({\textstyle{1\over 2}}-it)=e^{-2i\phi}\Delta({\textstyle{1\over 2}}+it),
$$
where the last equality is derived by conjugation. Two values for $\phi$ are of special interest here: for $\phi=0$ the roots of the equation $\Delta({1\over 2}+it)=e^{2i\phi}$ correspond to real values of $\zeta({1\over 2}+it)$, whereas $\phi={\pi\over 2}$ yields the purely imaginary values; this follows immediately from the fact that for such values of $t$
$$
\zeta({\textstyle{1\over 2}}+it)=\pm\zeta({\textstyle{1\over 2}}-it)=\pm\overline{\zeta({\textstyle{1\over 2}}+it)}.
$$
The roots $t$ of $\Delta({1\over 2}+it)-1$ are called Gram points after Gram \cite{gram} who observed that the first of those roots separate consecutive zeta zeros on the critical line; we shall discuss our results concerning this aspect in the last but one section. For the roots of $\Phi(t)$ which are no ordinates of zeros of the zeta-function we obviously have
$$
\arg\zeta({\textstyle{1\over 2}}+it)\equiv \phi\bmod \pi.
$$
Hence, the roots of $\Phi(t)$ correspond to the intersection points of the curve $\zeta({1\over 2}+i\R)$ with the straight line $e^{i\phi}\R$ through the origin. In this note we are mainly concerned about the non-zero intersection points. We expect that all but finitely many of the roots of the two factors of $\Phi(t)$ in (\ref{vieh}) are distinct (for each fixed value of $\phi$). Since the set of zeros of zeta is countable, there have to exist values of $\phi$ in the uncountable set $[0,\pi)$ for which $\Delta({1\over 2}+i\gamma)\neq e^{2i\phi}$ for all ordinates $\gamma$ of nontrivial zeros; however, we can even not exclude the possibility of $\Delta({1\over 2}+i\gamma)=1$ for a single ordinate $\gamma$.
\par

Denote by $N_\phi(T)$ the number of zeros of the function $\Phi(t)$ with $t\in(0,T]$, then
$$
N_\phi(T)=N_0(T)+N_\phi^\Delta(T),
$$
where $N_0(T)$ counts the nontrivial zeros of $\zeta(s)$ on the critical line with imaginary part in $(0,T]$ and $N_\phi^\Delta(T)$ is the number of roots of the equation $\Delta({1\over 2}+it)=e^{i\phi}$ with $t\in(0,T]$. All these functions are counting according multiplicities; it is easy to see that all roots counted by $N_\phi^\Delta$ are simple (see (\ref{delt})). By a deep result of Conrey \cite{conre} (building on work of Levinson \cite{levi}), more than two fifths of the zeros of $\zeta(s)$ lie on the critical line,
\begin{equation}\label{peter}
N_0(T)\geq {2\over 5}N(T)={T\over 5\pi}\log T+O(T),
\end{equation}
where $N(T)={T\over 2\pi}\log{T\over 2\pi e}+O(\log T)$ is the Riemann-von Mangoldt counting function for nontrivial zeros without specifying the real part (for this and for other basics from zeta-function theory we refer to Ivi\'c \cite{ivic}). First of all, we shall prove an asymptotic formula for the counting function $N_\phi^\Delta(T)$:

\begin{Theorem}\label{one}
For any $\phi\in[0,\pi)$, as $T\to\infty$,
$$
N_\phi^\Delta(T)={T\over 2\pi}\log{T\over 2\pi e}+O(\log T);
$$
in particular, the set $\zeta({1\over 2}+i\R)\cap e^{i\phi}\R$ is countable.
\end{Theorem}

\noindent If the Riemann hypothesis is true (or if almost all zeta zeros lie on the critical line $N(T)\sim N_0(T)$), then Theorem \ref{one} implies
\begin{equation}\label{conseq}
N_\phi^\Delta(T)\sim N_0(T)\qquad\mbox{and}\qquad N_\phi(T)\sim {T\over \pi}\log T;
\end{equation}
unconditionally, it follows that $T\log T$ is the correct order of growth for $N_\phi(T)$.
\par

Moreover, we shall prove asymptotic formulae for the associated first and second discrete moment:

\begin{Theorem}\label{two}
For any $\phi\in[0,\pi)$, as $T\to\infty$,
\begin{equation}\label{vienas}
\sum_{\substack{0<t\leq T\\\zeta({1\over 2}+it)\in e^{i\phi}\sR}}\zeta\left({\textstyle{\frac12}}+it\right)
=2e^{i\phi}\cos\phi\, \frac{T}{2\pi}\log\frac{T}{2\pi e}+O\left(T^{\frac12+\epsilon}\right),
\end{equation}
and
\begin{eqnarray}\label{du}
\sum_{\substack{0<t\leq T\\\zeta({1\over 2}+it)\in e^{i\phi}\sR}}\left|\zeta\left({\textstyle{\frac12}}+it\right)\right|^2&=&\frac{T}{2\pi}\left(\log\frac{T}{2\pi e}\right)^2+(2c+2\cos(2\phi))\frac{T}{2\pi}\log\frac{T}{2\pi e}\nonumber\\
&&+{T\over 2\pi}+O\left(T^{{1\over 2}+\epsilon}\right),
\end{eqnarray}
where $c:=\lim_{N\to\infty}({1\over N}\sum_{n=1}^N{1\over n}-\log N)=0.577\ldots$ is the Euler-Mascheroni constant.
\end{Theorem}

\noindent Note that the leading main term of the second moment is independent of $\phi$ whereas the first one vanishes for $\phi={\pi\over 2}$. These asymptotic formulae provide interesting information for the value-distribution of the zeta-function on the critical line. As an immediate consequence of (\ref{conseq}) in combination with (\ref{vienas}) we note

\begin{Corollary}\label{drei}
For $\phi\in[0,\pi)$,
\begin{eqnarray*}
e^{i\phi}\cos\phi &\leq & \liminf_{T\to\infty}{1\over N_\phi(T)}\sum_{\substack{0<t\leq T\\\zeta({1\over 2}+it)\in e^{i\phi}\sR}}\zeta\left({\textstyle{\frac12}}+it\right)\\
&\leq & \limsup_{T\to\infty}{1\over N_\phi(T)}\sum_{\substack{0<t\leq T\\\zeta({1\over 2}+it)\in e^{i\phi}\sR}}\zeta\left({\textstyle{\frac12}}+it\right)
\leq {\textstyle{10\over 7}}e^{i\phi}\cos\phi.
\end{eqnarray*}
If Riemann's hypothesis is true (or if almost all zeta zeros lie on the critical line), then the mean-value of the points in $\zeta({1\over 2}+i\R)\cap e^{i\phi}\R$ exists and is equal to
$$
\lim_{T\to\infty}{1\over N_\phi(T)}\sum_{\substack{0<t\leq T\\\zeta({1\over 2}+it)\in e^{i\phi}\sR}}\zeta\left({\textstyle{\frac12}}+it\right)=e^{i\phi}\cos\phi;
$$
for purely imaginary values the statement is unconditionally true.
\end{Corollary}

\noindent The case of purely imaginary values is special because formula (\ref{vienas}) gives for $\phi={\pi\over 2}$ just an upper bound for the first moment which leads to mean-value zero. For the real values of the zeta-function on the critical line the mean-value equals one provided Riemann's hypothesis is true. For any value of $\phi$, the real part of the according mean-value is non-negative (explaining Edwards' observation from above). These different mean-values according to $\phi$ are well reflected in the graph of the curve $t\mapsto\zeta({1\over 2}+it)$, see Figure 1 below: the grey circle stands for the means $e^{i\phi}\cos\phi$ of the values $\zeta({1\over 2}+i\R)$ on the straight line $e^{i\phi}\R$ as $\phi$ varies; the symmetry with respect to the real axis and the tendency for $\Re \zeta({1\over 2}+it)$ to be positive are nicely explained by (\ref{vienas}) and (\ref{du}).
\begin{figure}[ht]
\includegraphics[scale=.75]{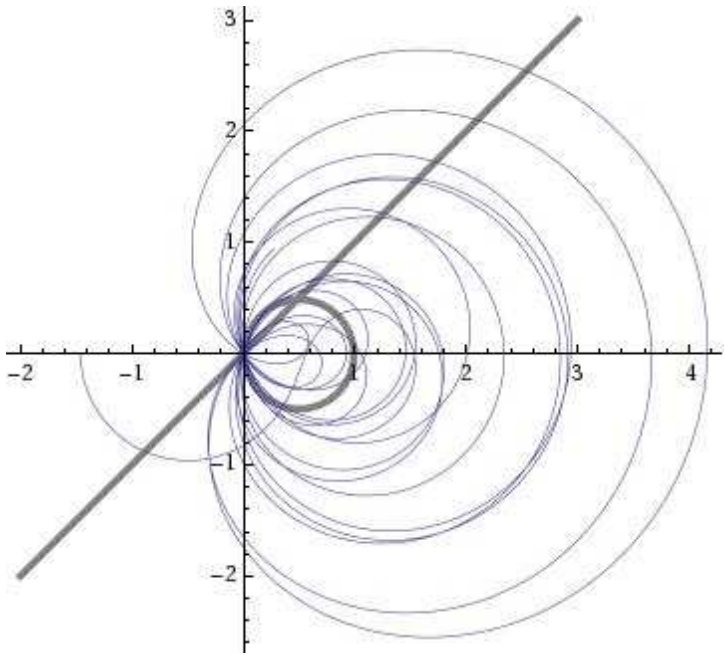}
\caption{The curve $t\mapsto\zeta({1\over 2}+it)$ for $t\in[0,70]$ in addition with the circle of the mean-values and the bisecting line of the first and third quadrant ($\phi={\pi\over 4}$); here the mean-value is ${1\over 2}(1+i)$, the intersection point of the grey circle with the bisecting line.}
\end{figure}

Our next application shows that the zeta-function takes arbitrarily large values on any half-line $e^{i\phi}\R$ through the origin in the right half-plane.

\begin{Corollary}\label{vier}
For any $\phi\in[0,\pi)$ there exist infinitely many real values $t\to\infty$ such that 
$$
e^{-i\phi}\zeta({\textstyle{1\over 2}}+it)\geq (\log t)^{1\over 2};
$$
more precisely, any interval $(T,2T]$ with sufficiently large $T$ contains such a value $t$. In the purely imaginary case ($\phi={\pi\over 2}$) there exist also values $t\to\ +\infty$ such that
$$
i\zeta({\textstyle{1\over 2}}+it)\geq (\log t)^{1\over 2}.
$$
\end{Corollary}

\noindent  In particular, $\zeta({1\over 2}+it)$ assumes arbitrarily large real values. To the best knowledge of the authors the various $\Omega$-results for the zeta-function in the literature so far (e.g. Soundararajan \cite{sou}) do not imply this corollary.
\par

Finally, we give a lower bound for the number of roots of $\Delta({1\over 2}+it)-e^{2i\phi}$ distinct from ordinates of zeta zeros:

\begin{Corollary}\label{funf}
For $\phi\neq {\pi\over 2}$, as $T\to\infty$,
$$
\sum_{\substack{0<t\leq T\\\Delta({1\over 2}+it)=e^{2i\phi}, \zeta({1\over 2}+it)\neq 0}}1\geq {\textstyle{2(\cos\phi)^2\over \pi}}T.
$$
\end{Corollary}

\noindent For $\phi={\pi\over 2}$, the case of purely imaginary values we do not have any non-trivial lower bound. It has been conjectured by Fujii \cite{fuj} that the values ${1\over 2\pi i}\log\Delta({1\over 2}+i\gamma_n)$ are uniformly distributed modulo one (in the sense of Weyl), where $\gamma_n$ denotes the $n$-th ordinate of the nontrivial zeros of $\zeta(s)$ in ascending order. A proof of this conjecture would certainly yield more information about the quantity estimated in the last corollary.
\medskip

The remaining parts of this article are organized as follows. The next two sections contain the proofs of Theorem \ref{one} and Theorem \ref{two}. For the sake of completeness we give the proofs of the corollaries in the fourth section. In the last but one section we discuss the special case $\phi=0$ and how our results provide new information about Gram points. Finally, we conclude with some remarks on the shape of the curve $t\mapsto\zeta({1\over 2}+it)$ and further applications of our method. 

\section{Proof of Theorem \ref{one}}

Recall that $\Delta({1\over 2}+it)$ is a complex number from the unit circle whenever $t\in\R$. Moreover, $\Delta'({\textstyle{1\over 2}}+it)$ is non-vansishing, which follows from the asymptotic formula
\begin{equation}\label{delt}
{\Delta'\over \Delta}(\sigma+it)=-\log{\vert t\vert\over 2\pi}+O(\vert t\vert^{-1})\qquad\mbox{for}\quad \vert t\vert\geq 1.
\end{equation}
Consequently, $\Delta({1\over 2}+it)$ is spinning on the unit circle around the origin in clockwise direction with increasing speed as $t\to\infty$. Moreover, it follows that there exists no proper real interval ${\mathcal I}$ such that $\zeta({1\over 2}+it)$ lies on a straight line $e^{i\phi}\R$ for all $t\in{\mathcal I}$. For the first, lets assume that 
\begin{equation}\label{condit}
\Delta({\textstyle{1\over 2}}+iT)=\Delta({\textstyle{1\over 2}})=1. 
\end{equation}
Then the number of roots of the equation $\Delta({1\over 2}+it)=e^{2i\phi}$ with $0\leq t\leq T$ is up to the sign equal to the winding number of the curve
$$
\eta\,:\, [0,1]\to \C,\quad \lambda\mapsto \eta(\lambda):=\Delta({\textstyle{1\over 2}}+i\lambda T).
$$
This yields
$$
-N_\phi^\Delta(T)={1\over 2\pi i}\int_\eta {\d s\over s}={T\over 2\pi}\int_0^1 {\Delta'\over \Delta}({\textstyle{1\over 2}}+i\lambda T)\d \lambda.
$$
In order to use (\ref{delt}) we divide the integration interval into two subintervals. Noting that there are only finitely many roots of $\Delta({1\over 2}+it)-e^{2i\phi}$ for $0<t\leq 1$, we find for the term with the integral on the right-hand side above
\begin{eqnarray*}
\lefteqn{{T\over 2\pi}\left\{\int_0^{1/T}+\int_{1/T}^1\right\}{\Delta'\over \Delta}({\textstyle{1\over 2}}+i\lambda T)\d \lambda}\\
&&=O(1)+{T\over 2\pi}\int_{1/T}^1\left(-\log{\lambda T\over 2\pi}+O((\lambda T)^{-1})\right)\d\lambda.
\end{eqnarray*}
Hence, the asymptotic formula for $N_\phi^\Delta(T)$ follows by integration; however, to get rid of our assumption (\ref{condit}) on $T$, by (\ref{delt}) we may substitute this by any $T$ at the expense of an error $O(\log T)$. This proves Theorem \ref{one}.

\section{Proof of Theorem \ref{two}}

The proof relies on a new variation of the method of Conrey, Ghosh \& Gonek (see, e.g., \cite{cgg}); for our purpose we shall work with the logarithmic derivative of $\Delta(s)-e^{2i\phi}$, before only logarithmic derivatives of Dirichlet series were treated.
\par\medskip

\subsection{Proof of the first moment} Since the nontrivial zeros of the zeta-function do not contribute to the sum in question, we only have to take the roots of $\Delta({1\over 2}+it)-e^{2i\phi}$ into account. Denoting those roots by $t_n^\phi$ in ascending order $0< t_1^\phi\leq t_2^\phi\leq \ldots$, we thus have
$$
\sum_{\substack{0<t\leq T\\\zeta({1\over 2}+it)\in e^{i\phi}\sR}}\zeta\left({\textstyle{\frac12}}+it\right)
=\sum_{\substack{0<t_n^\phi\leq T}}\zeta({\textstyle{\frac12}}+it).
$$
In view of Theorem \ref{one} the sequence of these $t_n^\phi$ cannot lie too dense. As a matter of fact, for any given $T_0$ there exists $T\in[T_0,T_0+1)$ such that
\begin{equation}\label{condi}
\min_{t_n^\phi}\vert T-t_n^\phi\vert\geq {1\over \log T},
\end{equation}
where the minimum is taken over all $t_n^\phi$. For $\vert t\vert\geq 10$, one has $\vert\Delta(s)\vert=1$ if and only if $\Re s={1\over 2}$ (see Spira \cite{spira}, resp. Dixon \& Schoenfeld \cite{ds} for a slight improvement). For smaller values of $t$ there are roots off the critical line, however, they from a sparse set. Hence, we may assume that there are no roots of $\Delta(s)-e^{2i\phi}$ on the rectangular contour ${\mathcal C}$ with corners $2+i,2+iT,1-a+iT,1-a+i$ in counterclockwise direction, where $a=1+\frac{1}{\log T}$ (if there are roots on the contour we make small indentions at the expense of a bounded error). By the calculus of residues,
\begin{eqnarray*}
\lefteqn{\sum_{0<t_n^\phi\leq T}\zeta({\textstyle{\frac12}}+it)}\\
&=&\frac{1}{2\pi  i}\left\{\int_{2+i}^{2+iT}+\int_{2+iT}^{1-a+iT}+\int_{1-a+iT}^{1-a+i}+\int_{1-a+i}^{2+i}\right\} \zeta(s)\frac{\Delta'(s)}{\Delta(s)-e^{2i\phi}}\d s+O(1)\\
&=&\sum_{i=1}^{4}\mathcal{I}_i+O(1),
\end{eqnarray*}
say. The bounded error term $O(1)$ results from at most finitely many possible residues outside ${\mathcal C}$, resp. contributions $\zeta({1\over 2}+it_n^\phi)$ with $0<t_n\leq 1$.
\par

First we consider $\mathcal{I}_1$. It is well-known that
\begin{equation}\label{L2}
\Delta(\sigma+it)=\left({\vert t\vert\over 2\pi}\right)^{{\textstyle{1\over 2}}-\sigma-it}\exp(i(t+{\textstyle{\pi\over 4}}))(1+O(\vert t\vert^{-1}))\qquad\mbox{for}\quad \vert t\vert\geq 1
\end{equation}
uniformly for any $\sigma$ from a bounded interval. Hence,
\begin{equation}\label{star}
{1\over \Delta(s)-e^{2i\phi}}={-e^{-2i\phi}\over 1-e^{-2i\phi}\Delta(s)}=-e^{-2i\phi}\left(1+\sum_{k=1}^\infty e^{-2ki\phi}\Delta(s)^k\right);
\end{equation}
in view of (\ref{L2}) the infinite geometric series on the right is converging very quickly for $s\in 2+i\R$. Writing $\Delta'(s)=\Delta(s){\Delta'\over \Delta}(s)$ and using (\ref{L2}) in combination with (\ref{delt}), we easily find
$$
{1\over 2\pi i}\int_{2+i}^{2+iT}\zeta(s)\Delta'(s)\sum_{k=1}^\infty e^{-2ki\phi}\Delta(s)^k\d s\ll 1.
$$
The same trick leads to 
$$
{1\over 2\pi i}\int_{2+i}^{2+iT}\zeta(s)\Delta'(s)\d s\ll 1,
$$
which implies $\mathcal{I}_1\ll 1$. The horizontal integral $\mathcal{I}_4$ is independent of $T$, so $\mathcal{I}_4\ll 1$. For $\mathcal{I}_2$ we notice that
the denominator $\Delta(s)-e^{i\phi}$ of the integrand is off the critical line is either dominated by $\Delta(s)$ for $\Re s<{1\over 2}$ or by $e^{i\phi}$ if $\Re s>{1\over 2}$. On the critical line we have 
$$
\Delta({\textstyle{1\over 2}}+iT)-e^{i\phi}=\Delta({\textstyle{1\over 2}}+iT)-\Delta({\textstyle{1\over 2}}+it_m^\phi)
$$
for some integer $m$. It follows from (\ref{L2}) and (\ref{condi}) that
$$
\Delta({\textstyle{1\over 2}}+iT)-e^{i\phi}\gg \vert T-t_m^\phi\vert \log T\gg 1.
$$
Now using (\ref{delt}) in combination with
\begin{equation}\label{ioi}
\zeta(\sigma+it)\ll 1+\vert t\vert^{{1\over 2}(1-\sigma)+\epsilon}\qquad\mbox{for}\quad 1-a\leq \sigma\leq 2,\ \vert t\vert\geq 1,
\end{equation}
we may deduce $\mathcal{I}_2\ll T^{\frac12+\epsilon}$.
\par

The main term of the asymptotic formula comes from the integral $\mathcal{I}_3$ which we have to evaluate now. Via $s\mapsto 1-\overline{s}$ we find
$$
\mathcal{I}_3=-\frac{1}{2\pi  i}\int_{a+i}^{a+iT}\zeta(1-\overline{s}) \frac{\Delta'(1-\overline{s})}{\Delta(1-\overline{s})-e^{2i\phi}}\d s.
$$
By the reflection principle,
$$
\overline{\mathcal{I}_3}=-\frac{1}{2\pi i}\int_{a+i}^{a+iT}\zeta(1-s)\frac{\Delta'(1-s)}{\Delta(1-s)-e^{-2i\phi}}\d s.
$$
In view of the functional equation (\ref{feq}) the integrand can be rewritten as
$$
\zeta(1-s)\frac{\Delta'(1-s)}{\Delta(1-s)-e^{-2i\phi}}=\zeta(s)\Delta'(1-s){1\over 1-e^{-2i\phi}\Delta(1-s)^{-1}}.
$$
Note that $\Delta(1-s)^{-1}=\Delta(s)$ (again by (\ref{feq})). With regard to (\ref{L2}) this function is small on $s=a+i\R$, and thus we may use the geometric series expansion to find 
\begin{eqnarray}\label{conj}
\overline{\mathcal{I}_3}&=&-\frac{1}{2\pi i} \int_{a+i}^{a+iT}\zeta(s)\Delta'(1-s) \left(1+\frac{e^{-2i\phi}}{\Delta(1-s)}+\sum_{k=2}^{\infty}e^{-2ki\phi}\Delta(s)^k\right)\d s\nonumber\\
&=&\mathcal{J}_1+\mathcal{J}_2+O\left(T^{\frac12+\epsilon}\right),
\end{eqnarray}
say. First we consider $\mathcal{J}_1$. According to (\ref{delt}) we have
$$
\mathcal{J}_1=\int_{1}^{T}\left(\log\frac{\tau}{2\pi}+O\left(\frac1{\tau}\right)\right)\d\left(\frac{1}{2\pi i}\int_{a+i}^{a+i\tau}\zeta(s)\Delta(1-s)\d s\right).
$$
For the inner integral we use Gonek's lemma:
\begin{Lemma}\label{gonekl}
Suppose that $\sum_{n=1}^{\infty}a(n)n^{-s}$ converges for $\sigma>1$ where $a(n)\ll \epsilon$ for any $\epsilon >0$. Then we have, uniformly for $1<a \leq 2$,
\begin{align*}
&
\frac{1}{2\pi i}\int_{a+i}^{a+iT}\left(\frac{m}{2\pi}\right)^s\Gamma(s)\exp\left(\delta \frac{\pi i s}{2}\right)\sum_{n=1}^{\infty}\frac{a(n)}{n^s}\d s\\
&\quad{}=
\left\{\begin{array}{ll}
\sum_{n\leq \frac{Tm}{2\pi}}a(n)\exp(-2\pi i\frac{n}{m})+O\left(m^a T^{a-\frac12+\epsilon}\right)&\mbox{ if } \quad \delta = -1,\\
O(m^a) & \mbox{ if }\quad \delta = +1.
\end{array}\right.
\end{align*}
\end{Lemma}
\begin{proof}
Because of the absolute convergence we may interchange the order of summation and integration. For the integral we use Lemma 1 from  \cite{gonekk} and for the sum Lemma 4 from \cite{gonekk2}. 
\end{proof}

This yields
$$
\frac{1}{2\pi i}\int_{a+i}^{a+i\tau}\zeta(s)\Delta(1-s)\d s={\tau\over 2\pi}+O(\tau^{{1\over 2}+\epsilon}),
$$
resp.
\begin{eqnarray*}
\mathcal{J}_1&=&\int_{1}^{T}\left(\log\frac{\tau}{2\pi}+O\left(\frac1{\tau}\right)\right)\d\left(
\frac{\tau}{2\pi}+O\left(\tau^{\frac12+\epsilon}\right)\right)\\
&=&\frac{T}{2\pi}\log\frac{T}{2\pi e}+O\left(T^{\frac12+\epsilon}\right).
\end{eqnarray*}
In order to estimate $\mathcal{J}_2$ we use (\ref{delt}) again and obtain
$$
\mathcal{J}_2= e^{-2i\phi}\int_{1}^{T}\left(\log\frac{\tau}{2\pi}+O\left(\frac1{\tau}\right)\right)\d\left(\frac{1}{2\pi i}\int_{a+i}^{a+i\tau}\zeta(s)\d s\right).
$$
Since $a=1+\frac{1}{\log T}$ we have 
$$
\frac{1}{2\pi i}\int_{a+i}^{a+i\tau}\zeta(s)\d s=\frac{1}{2\pi}\sum_{n=1}^{\infty}\frac{1}{n^a}
\int_{1}^{\tau}\frac{1}{n^{it}}\d t = \frac{\tau}{2\pi}+O(\log T),
$$
which leads to 
$$
\mathcal{J}_2=e^{-2i\phi}\frac{T}{2\pi}\log\frac{T}{2\pi e}+O\left((\log T)^2\right).
$$

Hence, after conjugation (because of (\ref{conj})) we arrive at
$$
\sum_{0<t_n^\phi\leq T}\zeta({\textstyle{\frac12}}+it)=(1+e^{2i\phi})\frac{T}{2\pi}\log\frac{T}{2\pi e}+O\left(T^{\frac12+\epsilon}\right).
$$
Note that
\begin{equation}\label{zvaigzde}
1+e^{2i\phi}=2e^{i\phi}\cos\phi,
\end{equation}
and so we have proved the first asymptotic formula (\ref{vienas}) for any value $T$, satisfying (\ref{condi}). To get this uniformly in $T$ we allow an arbitrarily $T$ at the expense of an error 
$$
\zeta({\textstyle{1\over 2}}+it)\ll T^{{1\over 4}+\epsilon};
$$
the latter estimate is a trivial consequence of the Phragm\'en-Lindel\"of principle applied to the functional equation. This proves formula (\ref{vienas}).
\medskip

\subsection{Proof of the second moment} We can procced in a rather similar way. Assuming the same as for the first moment, we have
\begin{align*}
&\sum_{0<t_n^\phi\leq T}\left|\zeta\left({\textstyle{\frac12}}+it\right)\right|^2\\
&\quad{}=\frac{1}{2\pi i}\left\{\int_{2+i}^{2+iT}+\int_{2+iT}^{1-a+iT}+\int_{1-a+iT}^{1-a+i}+\int_{1-a+i}^{2+i}\right\}\zeta(s)\zeta(1-s)\frac{\Delta'(s)}{\Delta(s)-e^{2i\phi}}\d s+O(1)\\
&\quad{}=\sum_{i=1}^{4}\mathcal{I}_i+O(1),
\end{align*}
say.
\par

First we consider $\mathcal{I}_1$. According to (\ref{delt}) and geometric progression we find similar to the previous case of the first moment
\begin{align*}
\mathcal{I}_1=&-e^{-2i\phi}\frac{1}{2\pi i}\int_{2+i}^{2+iT}\zeta(s)^2{\Delta'\over \Delta}(s)\left(1+\sum_{k=1}^{\infty}e^{-2ki\phi}\Delta(s)^k\right)\d s\\
=&e^{-2i\phi}\int_{1}^{T}\left(\log\frac{\tau}{2\pi}+O\left(\frac{1}{\tau}\right)\right)\d \left(\frac{1}{2\pi i}\int_{2+i}^{2+i\tau}\zeta(s)^2 \d s\right)+O(1).
\end{align*}
Since
$$
\frac{1}{2\pi i}\int_{2+i}^{2+i\tau}\zeta(s)^2 \d s=\frac{1}{2\pi}\sum_{m,n=1}^{\infty}\frac{1}{(mn)^2}
\int_{1}^{\tau}\frac{1}{(mn)^{it}}\d t = \frac{\tau}{2\pi}+O(1),
$$
we get
$$
\mathcal{I}_1=e^{-2i\phi}\frac{T}{2\pi}\log\frac{T}{2\pi e}+O(1).
$$
By the same argument as for the first moment we find $\mathcal{I}_2\ll T^{\frac12+\epsilon}, \mathcal{I}_4\ll 1$. Again, the integral $\mathcal{I}_3$ gives the main contribution. Via $s\mapsto 1-\overline{s}$ we find
$$
\overline{\mathcal{I}_3}=-\frac{1}{2\pi i}\int_{a+i}^{a+iT}\zeta(1-s)\zeta(s)\frac{\Delta'(1-s)}{\Delta(1-s)-e^{-2i\phi}}\d s.
$$
By the functional equation (\ref{feq}) and the infinite geometric series,
\begin{align*}
\overline{\mathcal{I}_3}&=-\frac{1}{2\pi i}\int_{a+i}^{a+iT}\zeta(s)^2\Delta'(1-s)\left(1+\frac{e^{-2i\phi}}{\Delta(1-s)}+\sum_{k=2}^{\infty}e^{-2ki\phi}\Delta(s)^k\right)\d s\\
&=\mathcal{J}_1+\mathcal{J}_2+O\left(T^{\frac12+\epsilon}\right).
\end{align*}
We start with $\mathcal{J}_1$. According to (\ref{delt}) we have
\begin{align*}
\mathcal{J}_1=&
\int_{1}^{T}\left(\log\frac{\tau}{2\pi}+O\left(\frac1{\tau}\right)\right)
\d\left(\frac{1}{2\pi i}\int_{a+i}^{a+i\tau}\zeta(s)^2\Delta(1-s)\d s\right).
\end{align*}
Using Gonek's Lemma \ref{gonekl} for the inner integral and a standard bound for the Dirichlet divisor problem (see \cite{ivic}), we find
\begin{eqnarray*}
\frac{1}{2\pi i}\int_{a+i}^{a+i\tau}\zeta(s)^2\Delta(1-s)\d s&=&\sum_{mn\leq \frac{\tau}{2\pi}}1+O\left(\tau^{\frac12+\epsilon}\right)\\
&=&\frac{\tau}{2\pi}\log\frac{\tau}{2\pi}+(2c-1)\frac{\tau}{2\pi}+O\left(\tau^{{1\over 2}+\epsilon}\right).
\end{eqnarray*}
This yields
\begin{align*}
\mathcal{J}_1=&
\int_{1}^{T}\left(\log\frac{\tau}{2\pi}+O\left(\frac1{\tau}\right)\right)\d\left(\frac{\tau}{2\pi}\log\frac{\tau}{2\pi}+(2c-1)\frac{\tau}{2\pi}+O\left(\tau^{{1\over 2}+\epsilon}\right)\right)\\
=&\frac{T}{2\pi}\left(\log\frac{T}{2\pi}\right)^2+(2c-2)\frac{T}{2\pi}\log\frac{T}{2\pi e}+O\left(T^{{1\over 2}+\epsilon}\right).
\end{align*}

Finally we consider $\mathcal{J}_2$. According to (\ref{delt}) we have
\begin{align*}
\mathcal{J}_2=&e^{-2i\phi}\int_{1}^{T}\left(\log\frac{\tau}{2\pi}+O\left(\frac1{\tau}\right)\right)\d\left(\frac{1}{2\pi i}\int_{a+i}^{a+i\tau}\zeta(s)^2\d s\right).
\end{align*}
Since $a=1+\frac{1}{\log T}$,
$$
\frac{1}{2\pi i}\int_{a+i}^{a+i\tau}\zeta(s)^2\d s=\frac{1}{2\pi}\sum_{m,n=1}^{\infty}\frac{1}{(mn)^a}
\int_{1}^{\tau}\frac{1}{(mn)^{it}}\d t = \frac{\tau}{2\pi}+O\left( (\log T)^2\right),
$$
which leads to
$$
\mathcal{J}_2=e^{-2i\phi}\frac{T}{2\pi}\log\frac{T}{2\pi e}+O\left((\log T)^3\right).
$$
Collecting together, formula (\ref{du}) of the theorem follows in just the same way as the first moment asymptotics. This finishes the proof of Theorem \ref{two}.

\section{Proof of the corollaries}

\subsection{Proof of Corollary \ref{drei}} If the Riemann hypothesis is true, $N_\phi(T)\sim {T\over \pi}\log T$ by (\ref{conseq}), hence the existence and the formula for the mean-value follows immediately from (\ref{vienas}). If $\phi={\pi\over 2}$ the main term of (\ref{vienas}) vanishes and, using the unconditional bound $N_\phi(T)\geq N_\phi^\Delta(T)$ in place of the unproved Riemann hypothesis, we get in this case
\begin{equation}\label{uui}
{1\over N_\phi(T)}\sum_{\substack{0<t\leq T\\0\neq \zeta({1\over 2}+it)\in i\sR}}\zeta\left({\textstyle{\frac12}}+it\right)
\ll T^{-\frac12+\epsilon};
\end{equation}
consequently, the mean-value is zero for $\phi={\pi\over 2}$. Unconditionally, one has to use (\ref{peter}) to derive the bounds for $\liminf$ and $\limsup$ as in the corollary.

\subsection{Proof of Corollary \ref{vier}} By Theorem \ref{one} and \ref{two},
$$
N:=N_\phi^\Delta(2T)-N_\phi^\Delta(T)\sim {T\over 2\pi}\log T
$$
and
\begin{equation}\label{u}
\Sigma:=\sum_{\substack{T<t\leq 2T\\\zeta({1\over 2}+it)\in e^{i\phi}\sR}}\left|\zeta\left({\textstyle{\frac12}}+it\right)\right|^2\sim \frac{T}{2\pi}(\log T)^2.
\end{equation}
Suppose that $\vert\zeta({1\over 2}+it)\vert<c(\log T)^\alpha$ with some positive constant $c$ for all terms in the latter sum, then we obtain the trivial bound
$$
\Sigma<c^2 (\log T)^{2\alpha}N\sim c^2{T\over 2\pi}(\log T)^{1+2\alpha}.
$$
In order to have no contradiction with the right-hand side of (\ref{u}), we find $\alpha\leq {1\over 2}$. Consequently, in any interval $(T,2T]$ for sufficiently large $T$, there exists a value $t$ such that $\zeta({1\over 2}+it)\in e^{i\phi}\R$ and $\vert\zeta({1\over 2}+it)\vert \geq(\log t)^{1\over 2}$. If all those values $e^{-i\phi}\zeta({1\over 2}+it)$ would be negative, we would get a contradiction to (\ref{vienas}). In the purely imaginary case we get large values with both signs according to (\ref{uui}). This proves the corollary.

\subsection{Proof of Corollary \ref{funf}} Applying the Cauchy-Schwarz inequality,
$$
\Big\vert\sum_{\substack{0<t\leq T\\\zeta({1\over 2}+it)\in e^{i\phi}\sR}}\zeta\left({\textstyle{\frac12}}+it\right)\Big\vert^2\leq \Big( \sum_{\substack{0<t\leq T\\\Delta({1\over 2}+it)=e^{2i\phi}, \zeta({1\over 2}+it)\neq 0}}1\Big)\Big(\sum_{\substack{0<t\leq T\\\zeta({1\over 2}+it)\in e^{i\phi}\sR}}\vert\zeta\left({\textstyle{\frac12}}+it\right)\vert^2\Big).
$$
Inserting the asymptotic formulae of Theorem \ref{two} yields the assertion of the corollary.

\section{Gram points} 

The roots of the equation $\Delta({1\over 2}+it)=1$ constitute the so-called Gram points $t_n$ which are usually defined by the roots of the equation $\vartheta(t)=\pi n$ for $n\in\N$, where $\exp(i\vartheta(t))=\Delta({1\over 2}+it)^{-{1\over 2}}$. The function
$$
Z(t):=\exp(i\vartheta(t))\zeta({\textstyle{1\over 2}}+it)
$$
is real-valued for real $t$ and Gram's law states that 
$$
(-1)^nZ(t_n)>0.
$$
In fact, each sign change of the real valued function $Z(t)$ corresponds to a zero and the number of Gram points matches with the number of nontrivial zeros; by Theorem \ref{one} we have $N_\phi^\Delta(T)-N(T)\ll \log T$. As Gram observed in 1903 for the first zeta zeros, consecutive ordinates of nontrivial zeros are separated by a Gram point and vice versa; it should be mentioned that Gram himself seemed to doubt that this separation would persist indefinitely (cf. \cite{edwa}; page 127). The first failure of Gram's law appears for $n=126$, as discovered by Hutchinson \cite{hutch}. It was first shown by Titchmarsh \cite{titch} that Gram's law is violated infinitely often, and, recently, Trudgian \cite{trudg} has proved that it fails for a postive proportion of values of $n$ as $n\to\infty$. Since $\zeta({1\over 2}+it)<0$ if and only if $t$ is a Gram point for which Gram's law is not true, it follows that there exist infinitely many intersections of the curve $t\mapsto\zeta({1\over 2}+it)$ with the negative real axis and we may ask whether those points all lie in a bounded interval or not. This seems to be a difficult question which is of particular interest with respect to the conjectured denseness of $\zeta({1\over 2}+i\R)$ in the complex plane. 
\par

Our results have also an impact on previous research on Gram points. Titchmarsh \cite{titch} proved
$$
\sum_{n=M+1}^N Z(t_n)Z(t_{n+1})\sim -2(c+1)N
$$
for any $M$ as $N\to\infty$, which implies already the existence of infinitely many nontrivial zeros on the critical line; moreover, he conjectured
$$
\sum_{n=M+1}^N Z(t_n)^2Z(t_{n+1})^2\ll N(\log N)^A
$$
for some constant $A\geq 0$. The latter estimate was proved by Moser \cite{moser} with $A=5$. For this aim he showed the stronger estimate
$$
\sum_{n=M+1}^N Z(t_n)^4\ll N(\log N)^4.
$$
Next we shall derive a lower bound for this expression. Using  
$$
n={1\over \pi}\vartheta(t_n)\sim {1\over 2\pi}t_n\log t_n,
$$
the asymptotic formulae from Theorem \ref{two} translate to
\begin{equation}\label{mosel}
\sum_{n\leq N} Z(t_n)\sim 2N\qquad\mbox{and}\qquad \sum_{n\leq N}Z(t_n)^2\sim N\log N;
\end{equation}
the first formula above was already found by Titchmarsh \cite{titch} (with a different method). By the Cauchy-Schwarz inequality,
$$
\Big(\sum_{n\leq N}Z(t_n)^2\Big)^2\leq \Big(\sum_{n\leq N}Z(t_n)^4\Big)\Big(\sum_{n\leq N}1\Big).
$$
Substituting the asymptotic formula (\ref{mosel}) implies

\begin{Corollary}
As $N\to\infty$,
$$
\sum_{n\leq N}Z(t_n)^4 \geq (1+o(1))N(\log N)^2.
$$
\end{Corollary}

\section{Concluding Remarks}

We conclude with a few heuristical remarks. In Corollary \ref{drei} we have proved the existence and explicit values for the mean of $\zeta({1\over 2}+it)$ on straight lines $e^{i\phi}\R$ passing through the origin. Their positivity has confirmed an observation by Edwards (mentioned in the introduction). The symmetry with respect to the real line is well reflected in the vanishing main term of the first moment (\ref{vienas}) for $\phi={\pi\over 2}$ whereas the non-vanishing of the second moment (\ref{du}) implies the existence of infinitely many non-zero values $\zeta({1\over 2}+it)$ on the imaginary axis. Next we compute from Corollary \ref{drei}, under assumption of the Riemann hypothesis, the mean of all individual mean-values by integration over $\phi$:
$$
{1\over \pi}\int_0^\pi \lim_{T\to\infty}{1\over N_\phi(T)}\sum_{\substack{0<t\leq T\\\zeta({1\over 2}+it)\in e^{i\phi}\sR}}\zeta\left({\textstyle{\frac12}}+it\right)\d\phi ={1\over \pi}\int_0^\pi {\textstyle{1\over 2}}(1+e^{2i\phi})\d\phi={\textstyle{1\over 2}};
$$
here we have used (\ref{zvaigzde}) for expressing the individual mean-value from Corollary \ref{drei} in a more convenient expression for the integration with respect to $\phi$. This is of special interest with respect to another discrete moment considered by Garunk\v stis \& Steuding recently. In \cite{gs}, they have proved that, for any fixed complex number $a$, as $T\to\infty$,
\begin{eqnarray*}
\sum_{0<\gamma_a\leq T} \zeta'(\rho_a)&=&({\textstyle{1\over 2}}-a){T\over 2\pi}\left(\log {T\over 2\pi}\right)^2+2(c_0-1+a){T\over 2\pi}\log {T\over 2\pi}\\
&&+2(c_1-c_0-a){T\over 2\pi}+E(T),
\end{eqnarray*}
where the summation is over nontrivial roots $\rho_a=\beta_a+i\gamma_a$ of the equation $\zeta(s)=a$, the numbers $c_n$ are the Stieltjes constants (defined by the Laurent expansion of $\zeta(s)$ around $s=1$), and the error term is $E(T)\ll T\exp(-C(\log T)^{1\over 3})$ with some absolute positive constant $C$; if the Riemann hypothesis is true, then $E(T)\ll T^{{1\over 2}+\epsilon}$. For $a={1\over 2}$ the main term is of lower order which is related to the curve $t\mapsto\zeta({1\over 2}+it)$ not passing as often as for any other value of $a$ through the point ${1\over 2}$ in the complex plane. Roughly speaking, it seems that for any $\phi$ the set $\zeta({1\over 2}+i\R)\cap e^{i\phi}\R$ splits up into two subsets of approximately equal size consisting of elements having either very small or large absolute value; the mean $e^{i\phi}\cos\phi$ separates these sets and the quantity ${1\over 2}$ is the average over all these mean-values.
\medskip

It is easy to extend the obtained results to Dirichlet $L$-functions (with very similar results), whereas the case of higher degree $L$-functions seems to be rather difficult to handle. This is already visible in the case of Dedekind zeta-functions $\zeta_{\sK}$ to quadratic number fields for which Conrey, Ghosh \& Gonek \cite{cgg} succeeded to prove the existence of infinitely many simple zeros by factoring $\zeta_{\sK}=\zeta L$, where $L$ is the Dirichlet $L$-function to the character associated with $\K$, and using the symmetry of the zeros of the zeta-function.
\medskip

\noindent {\bf Acknowledgements.} The first named author was supported by an STIBET Contact Fellowship of the Graduate Schools at W\"urzburg University; he is grateful for this encouragement and wants to thank in particular Dr. Stephan Schr\"oder-K\"ohne.

\bigskip

\ \bigskip

\tiny

\noindent\parbox{7cm}{
Justas Kalpokas\\
Faculty of Mathematics and Informatics\\ 
Vilnius University\\
Naugarduko 24, 03225 Vilnius, Lithuania\\
justas.kalpokas@mif.vu.lt}
\hfill
\parbox{7cm}{\noindent
Justas Kalpokas\\
Faculty of Social Informatics\\
Mykolas Romeris University\\
Ateities 20, 08303 Vilnius, Lithuania\\
justas.kalpokas@mruni.lt
}
\medskip

\noindent
J\"orn Steuding\\
Department of Mathematics, W\"urzburg University\\ 
Am Hubland, 97\,218 W\"urzburg, Germany\\
steuding@mathematik.uni-wuerzburg.de


\small

\begin{thebibliography}{9}

\bibitem{bohr}{\sc H. Bohr, R. Courant}, Neue Anwendungen der Theorie der diophantischen Approximationen auf die Riemannsche Zetafunktion, {\it J. reine angew. Math.} {\bf 144} (1914), 249-274

\bibitem{conre}{\sc J.B. Conrey}, More than two fifths of the zeros of the Riemann zeta-function are on the critical line, {\it J. reine angew. Math.} {\bf 399} (1989), 1-26

\bibitem{cgg}{\sc J.B. Conrey, A. Ghosh, S.M. Gonek}, Simple zeros of the zeta function of a quadratic number field. I, {\it Invent. Math.} {\bf 86} (1986), 563-576

\bibitem{gonekk}{\sc J.B. Conrey, A. Ghosh,  S.M. Gonek}, Simple zeros of the Riemann zeta-function, {\it Proc. London Math. Soc.} {\bf 76(3)} (1998), 497-522

\bibitem{ds}{\sc R.D. Dixon, L. Schoenfeld}, The size of the Riemann zeta-function at places symmetric with respect to the point $1/2$, {\it Duke Math. J.} {\bf  33} (1966), 291-292

\bibitem{edwa}{\sc H.M. Edwards}, {\it Riemann's zeta function}, Academic Press 1974

\bibitem{fuj}{\sc A. Fujii}, Gram's law for the zeta zeros and the eigenvalues of Gaussian unitary ensembles, {\it Proc. Japan Acad.} {\bf 63} (1987), 392-395

\bibitem{gs}{\sc R. Garunk\v stis, J. Steuding}, On the roots of the equation $\zeta(s)=a$, submitted

\bibitem{gonekk2}{\sc S.M. Gonek}, Mean values of the Riemann zeta-function and its derivatives, {\it Invent. Math.} {\bf 75} (1984), 123-141

\bibitem{gram}{\sc J. Gram}, Sur les z\'eros de la fonction $\zeta(s)$ de Riemann, {\it Acta Math.} {\bf 27} (1903), 289-304 

\bibitem{hutch}{\sc J.I. Hutchinson}, On the roots of the Riemann zeta function, {\it Trans. A.M.S.} {\bf 27} (1925), 49-60

\bibitem{ivic}{\sc A. Ivi\'c}, {\it The Riemann zeta-function}, John Wiley \& Sons, New York 1985

\bibitem{joyner}{\sc D. Joyner}, {\it Distribution theorems of $L$-functions}, Pitman Research Notes in Mathematics, 1986

\bibitem{levi}{\sc N. Levinson}, More than one third of Riemann's zeta-function are on $\sigma={1\over 2}$, {\it Adv. Math.} {\bf 13} (1974), 383-436

\bibitem{moser}{\sc J. Moser}, The proof of the Titchmarsh hypothesis in the theory of the Riemann zeta-function, {\it Acta Arith.} {\bf 36} (1980), 147-156 (Russian)

\bibitem{sou}{\sc K. Soundararajan}, Extreme values of zeta and L-functions, preprint available as {\it arXiv:0708.3990}

\bibitem{spira}{\sc R. Spira}, An inequality for the Riemann zeta function, {\it Duke Math. J.} {\bf 32} (1965), 247-250

\bibitem{titch}{\sc E.C. Titchmarsh}, On van der Corput's method and the zeta-function of Riemann, IV, {\it Quart. J. Math.} {\bf 5} (1934), 98-105

\bibitem{trudg}{\sc T. Trudgian}, Gram's law fails a positive proportion of the time, submitted to the arXiv

\end{thebibliography}
\end{document}